\documentclass[12pt,reqno,oneside]{amsart}
\usepackage{amsfonts}
\usepackage{amsmath, amssymb}
\usepackage{hyperref}

\numberwithin{equation}{section} 
\newtheorem{theorem}{Theorem}[section]
\newtheorem{corollary}[theorem]{Corollary}
\newtheorem{lemma}[theorem]{Lemma}
\theoremstyle{definition}
\newtheorem{definition}[theorem]{Definition}
\theoremstyle{remark}
\newtheorem{remark}[theorem]{Remark}
\numberwithin{equation}{section}

\input{tcilatex}

\begin{document}
\title[Shell-like curves connected with Fibonacci numbers]{Certain classes  of
	bi-univalent functions related to Shell-like curves connected with Fibonacci numbers}
\author{N. Magesh,\; V. K. Balaji and C. Abirami}
\address{Post-Graduate and Research Department of Mathematics,\\
	Government Arts College for Men,\\
	Krishnagiri 635001, Tamilnadu, India.
	\texttt{e-mail:} $nmagi\_2000@yahoo.co.in$}
\address{ Department of Mathematics, L.N. Govt College, \\
	Ponneri, Chennai, Tamilnadu, India.
	\texttt{e-mail:} $balajilsp@yahoo.co.in$}
\address{Faculty of Engineering and Technology, SRM University, Kattankulathur-603203, Tamilnadu, India.
	\texttt{e-mail:} $shreelekha07@yahoo.com$}
\date{}

\begin{abstract}
Recently, in their pioneering work on the subject of bi-univalent functions,
Srivastava et al. \cite{HMS-AKM-PG} actually revived the study of the
coefficient problems involving bi-univalent functions. Inspired by the
pioneering work of Srivastava et al. \cite{HMS-AKM-PG}, there has been
triggering interest to study the coefficient problems for the different
subclasses of bi-univalent functions. 
Motivated largely by Ali et al. \cite{Ali-Ravi-Ma-Mina-class}, Srivastava et al. 
\cite{HMS-AKM-PG} and G\"{u}ney et al. \cite{HOG-GMS-JS-Fib-2018} in this 
paper, we consider certain classes of bi-univalent functions related to 
shell-like curves connected with Fibonacci numbers to obtain the estimates 
of second, third Taylor-Maclaurin coefficients and Fekete - Szeg\"{o} 
inequalities. Further, certain special cases are also indicated. Some 
interesting remarks of the results presented here are also discussed.
\end{abstract}

\maketitle


\section{Introduction and definitions}
Let $\mathbb{R}=\left( -\infty ,\infty \right) $\ be the set of real
numbers, $\mathbb{C}$ be the set of complex numbers and%
\begin{equation*}
\mathbb{N}:=\left\{ 1,2,3,\ldots \right\} =\mathbb{N}_{0}\backslash \left\{
0\right\}
\end{equation*}%
be the set of positive integers. Let $\mathcal{A}$ denote the class of functions of the form
\begin{equation}
f(z)=z+\sum\limits_{n=2}^{\infty }a_{n}z^{n}  \label{Int-e1}
\end{equation}%
which are analytic in the open unit disk $\mathbb{D}=\{z:z\in \mathbb{C}\,\,%
\mathrm{and}\,\,|z|<1\}.$ Further, by $\mathcal{S}$ we shall denote the
class of all functions in $\mathcal{A}$ which are univalent in $\mathbb{D}.$

\par Let $\mathcal{P}$ denote the class of functions of the form
\begin{align*}
p(z)=1+p_1z+p_2z^2+p_3z^3+\ldots \qquad (z \in \mathbb{D})
\end{align*}
which are analytic with $\Re~\{p(z)\}>0.$ Here $p(z)$ is called as Caratheodory functions \cite{PLD}. It is well known that the following correspondence between the class $\mathcal{P}$ and the class of Schwarz functions $w$ exists: $p \in \mathcal{P}$ if and only if $p(z)={1+w(z)}\left/\right.{1-w(z)}.$ Let $\mathcal{P}(\beta),$ $0 \leq \beta < 1,$ denote the class of analytic functions $p$ in $\mathbb{D}$ with $p(0)=1$ and $\Re\left\{p(z)\right\} > \beta.$ 

\par For analytic functions $f$ and $g$ in $\mathbb{D},$ $f$ is said to be
subordinate to $g$ if there exists an analytic function $w$ such that 
\begin{equation*}
w(0)=0, \quad \quad |w(z)|<1 \quad \mathrm{and} \quad f(z)=g(w(z)) \qquad
(z\in\mathbb{D}).
\end{equation*}
This subordination will be denoted here by
\begin{equation*}
f \prec g \qquad (z\in\mathbb{D})
\end{equation*}
or, conventionally, by
\begin{equation*}
f(z) \prec g(z) \qquad (z\in\mathbb{D}).
\end{equation*}

In particular, when $g$ is univalent in $\mathbb{D},$
\begin{equation*}
f \prec g \qquad (z\in\mathbb{D}) ~\Leftrightarrow ~f(0)=g(0) \quad \mathrm{%
	and} \quad f(\mathbb{D}) \subset g(\mathbb{D}).
\end{equation*}
Some of the important and well-investigated subclasses of the univalent
function class $\mathcal{S}$ include (for example) the class $\mathcal{S}%
^{\ast }(\alpha )$ of starlike functions of order $\alpha $ $(0\leqq \alpha
<1)$ in $\mathbb{D}$ and the class $\mathcal{K}(\alpha )$ of convex
functions of order $\alpha $ $(0\leqq \alpha <1)$ in $\mathbb{D},$ 
the class $\mathcal{S}^{\ast }(\varphi )$ of Ma-Minda starlike functions and
the class $\mathcal{K}(\varphi )$ of Ma-Minda convex functions ($\varphi $
is an analytic function with positive real part in $\mathbb{D}$, $\varphi
(0)=1,$ $\varphi ^{\prime }(0)>0$ and $\varphi $ maps $\mathbb{D}$ onto a
region starlike with respect to 1 and symmetric with respect to the real
axis) (see \cite{PLD}). 

It is well known that every function $f\in \mathcal{S}$ has an inverse $%
f^{-1},$ defined by
\begin{equation*}
f^{-1}(f(z))=z \qquad (z\in\mathbb{D})
\end{equation*}
and
\begin{equation*}
f(f^{-1}(w))=w \qquad (|w| < r_0(f);\,\, r_0(f) \geqq \frac{1}{4}),
\end{equation*}
where
\begin{equation*}
f^{-1}(w) = w - a_2w^2 + (2a_2^2-a_3)w^3 - (5a_2^3-5a_2a_3+a_4)w^4+\ldots .
\end{equation*}

A function $f\in \mathcal{A}$ is said to be bi-univalent in $\mathbb{D}$ if
both $f(z)$ and $f^{-1}(z)$ are univalent in $\mathbb{D}.$ Let $\Sigma $
denote the class of bi-univalent functions in $\mathbb{D}$ given by (\ref%
{Int-e1}). 
Recently, in their pioneering work on the subject of bi-univalent functions,
Srivastava et al. \cite{HMS-AKM-PG} actually revived the study of the
coefficient problems involving bi-univalent functions. Various subclasses of
the bi-univalent function class $\Sigma $ were introduced and non-sharp
estimates on the first two coefficients $|a_{2}|$ and $|a_{3}|$ in the
Taylor-Maclaurin series expansion (\ref{Int-e1}) were found in several
recent investigations (see, for example, \cite%
{Ali-Ravi-Ma-Mina-class,SA-SY-BMM-2018,Bulut,Caglar-Orhan,Deniz,VBG-SBJ-Ganita-2018,Jay-SGH-SAH-2014,HMS-DB-2015,HMS-Caglar,HMS-SSE-RMA-2015,HMS-SG-FG-2015,HMS-SG-FG-2016,HMS-NM-JY-2014,HMS-GMS-NM-2013,Tang,Zaprawa}
and references therein). The afore-cited all these papers on the subject were
actually motivated by the pioneering work of Srivastava et al. \cite%
{HMS-AKM-PG}. However, the problem to find the coefficient bounds on $|a_{n}|
$ ($n=3,4,\dots $) for functions $f\in \Sigma $ is still an open problem.

The classes $\mathcal{S}\mathcal{L}(\tilde{p})$ and $\mathcal{K}\mathcal{S}\mathcal{L}(\tilde{p})$ of shell-like functions and convex shell-like functions are respectively, characterized by $zf^{\prime}\left/\right.f(z)\prec \tilde{p}(z)$ or $1+z^{2}f^{\prime\prime}\left/\right.f^{\prime}(z)\prec \tilde{p}(z),$ where $\tilde{p}(z)={(1+\tau^{2}z^{2})}\left/\right.{(1-\tau z - \tau^{2}z^{2})},$ $\tau = {(1-\sqrt{5})}\left/\right.{2} \approx -0.618.$ The classes $\mathcal{S}\mathcal{L}(\tilde{p})$ and $\mathcal{K}\mathcal{S}\mathcal{L}(\tilde{p})$ were introduced and studied by Sok\'{o}{\l} \cite{JS-1999} and Dziok et al. \cite{JD-RKR-JS-CMA-2011} respectively (see also \cite{JD-RKR-JS-AMC-2011,RKR-JS-2016}). The function $\tilde{p}$ is not univalent in $\mathbb{D},$ but it is univalent in the disc $|z| < {(3-\sqrt{5})}\left/\right.{2} \approx 0.38.$ For example, $\tilde{p}(0) = \tilde{p}\left({-1}\left/\right.{2\tau}\right)=1$ and 
$\tilde{p}\left(e^{\mp}\arccos\left(1/4\right)\right)={\sqrt{5}}\left/\right.{5}$ and it may also be noticed that ${1}\left/\right.{\left|\tau\right|} = {\left|\tau\right|}\left/\right.{1-\left|\tau\right|}$ which shows that the number $\left|\tau\right|$ divides $\left[0,\; 1\right]$ such that it fulfills the golden section. The image of the unit circle $\left|z\right|=1$ under $\tilde{p}$ is a curve described by the equation given by $\left(10x-\sqrt{5}\right)y^{2} = \left(\sqrt{5}-2x\right)\left(\sqrt{5}x-1\right)^{2},$ which is translated and revolved trisectrix of Maclaurin. The curve $\tilde{p}\left(re^{it}\right)$ is a closed curve without any loops for $0<r\leq r_{0}={(3-\sqrt{5})}\left/\right.{2} \approx 0.38.$ For $r_{0}<r<1,$ it has a loop and for $r=1$, it has a vertical asymptote. Since $\tau$ satisfies the equation $\tau^{2}=1+\tau,$  this expression can be used to obtain higher powers $\tau^{n}$ as a linear function of lower powers, which in turn can be decomposed all the way down to a linear combination of $\tau$ and $1.$ The resulting recurrence relationships yield Fibonacci numbers $u_{n}$
\[
\tau^{n} = u_{n}\tau+u_{n-1}.
\]
Recently Raina and Sok\'{o}{\l} \cite{RKR-JS-2016}, taking $\tau z =t,$ showed that
\begin{eqnarray}
\tilde{p}(z) 
&=& \dfrac{1+\tau^{2}z^{2}}{1-\tau z - \tau^{2}z^{2}} 
= 1 + \sum\limits_{n=2}^{\infty} \left(u_{n-1}+u_{n+1}\right)\tau^{n}z^{n},
\end{eqnarray}
where
\begin{eqnarray}
u_{n}=\dfrac{(1-\tau)^{n}-\tau}{\sqrt{5}},\qquad \tau = \dfrac{1-\sqrt{5}}{2}, ; \qquad n=1,\; 2,\; \ldots.
\end{eqnarray}
This shows that the relevant connection of $\tilde{p}$ with the sequence of Fibonacci numbers $u_{n}$, such that 
\[
u_{0} = 0,\quad u_{1}=1, \quad u_{n+2} = u_{n}+u_{n+1}
\]
for $n=0,\; 1,\; 2,\; 3,\; \ldots~.$ Hence
\begin{eqnarray}
\tilde{p}(z) 
&=& 1 + \tau z + 3\tau^{2}z^{2} + 4\tau^{3}z^{3}+7\tau^{4}z^{4}+11\tau^{5}z^{5}+\ldots.   
\end{eqnarray}
We note that the function $\tilde{p}$ belongs to the class $P(\beta)$ with $\beta = \dfrac{\sqrt{5}}{10} \approx 0.2236$ (see \cite{RKR-JS-2016}).

Motivated by works of Ali et al. \cite{Ali-Ravi-Ma-Mina-class}, G\"{u}ney et al. \cite{HOG-GMS-JS-Fib-2018} and Orhan et al. \cite{HO-NM-VKB-TMJ-2017}, we introduce the following new subclasses of bi-univalent functions, as follows.

\begin{definition}
\label{def1.1} A function $f\in \Sigma$ of the form 
\begin{equation*}
f(z) = z + \sum\limits_{n=2}^{\infty} a_{n} z^{n}, 
\end{equation*}
belongs to the class $\mathcal{W}\mathcal{S}\mathcal{L}_{\Sigma}(\gamma, \lambda, \alpha, \tilde{p}),
$ $\gamma\in \mathbb{C}\backslash\{0\},$ $\alpha \geq 0$ and $\lambda \geq 0,$  if the following conditions are satisfied: 
\begin{equation}  \label{CR-RAP-NM-P1-e1}
1+\frac{1}{\gamma}\left((1-\alpha+2\lambda)\frac{f(z)}{z}+(\alpha-2%
\lambda)f^{\prime }(z)+\lambda zf^{\prime \prime }(z)-1\right) \prec \widetilde{p(z)}=\dfrac{1+\tau^{2}z^{2}}{1-\tau z - \tau^{2}z^{2}},\;z\in\mathbb{D}
\end{equation}
and for $g(w)=f^{-1}(w)$ 
\begin{equation}  \label{CR-RAP-NM-P1-e2}
1+\frac{1}{\gamma}\left((1-\alpha+2\lambda)\frac{g(w)}{w}+(\alpha-2%
\lambda)g^{\prime }(w)+\lambda wg^{\prime \prime }(w)-1\right) \prec \widetilde{p(w)}=\dfrac{1+\tau^{2}w^{2}}{1-\tau w - \tau^{2}w^{2}},\;w\in\mathbb{D},
\end{equation}
where
$\tau = \dfrac{1-\sqrt{5}}{2} \approx -0.618.$
\end{definition}

It is interesting to note that the special values of $\alpha,$ $\gamma$ and $\lambda$ lead the class $\mathcal{W}\mathcal{S}\mathcal{L}_{\Sigma}(\gamma,
\lambda, \alpha, \tilde{p})$ to various subclasses, we illustrate the
following subclasses:

\begin{enumerate}
\item For $\alpha=1+2\lambda,$ we get the class $\mathcal{W}\mathcal{S}\mathcal{L}
_{\Sigma}(\gamma, \lambda, 1+2\lambda, \tilde{p})\equiv \mathcal{F}\mathcal{S}\mathcal{L}_{\Sigma}(\gamma, \lambda, \tilde{p}).$ A function $f\in \Sigma$ of the form 
\begin{equation*}
f(z) = z + \sum\limits_{n=2}^{\infty} a_{n} z^{n}, 
\end{equation*}
is said to be in $\mathcal{F}\mathcal{S}\mathcal{L}_{\Sigma}(\gamma, \lambda, \tilde{p}),$ if the
following conditions 
\begin{equation*}
1+\frac{1}{\gamma}\left(f^{\prime }(z)+\lambda zf^{\prime \prime
}(z)-1\right) \prec \widetilde{p(z)}=\dfrac{1+\tau^{2}z^{2}}{1-\tau z - \tau^{2}z^{2}},\;z\in\mathbb{D} 
\end{equation*}
and for $g(w)=f^{-1}(w)$ 
\begin{equation*}
1+\frac{1}{\gamma}\left(g^{\prime }(w)+\lambda wg^{\prime \prime
}(w)-1\right) \prec \widetilde{p(w)}=\dfrac{1+\tau^{2}w^{2}}{1-\tau w - \tau^{2}w^{2}},\;w\in\mathbb{D}, 
\end{equation*}
hold, where
$\tau = \dfrac{1-\sqrt{5}}{2} \approx -0.618.$

\item For $\lambda=0,$ we obtain the class $\mathcal{W}\mathcal{S}\mathcal{L}_{\Sigma}(\gamma, 0,
\alpha, \tilde{p})\equiv \mathcal{B}\mathcal{S}\mathcal{L}_{\Sigma}(\gamma, \alpha, \tilde{p}).$ A
function $f\in \Sigma$ of the form 
\begin{equation*}
f(z) = z + \sum\limits_{n=2}^{\infty} a_{n} z^{n}, 
\end{equation*}
is said to be in $\mathcal{B}\mathcal{S}\mathcal{L}_{\Sigma}(\gamma, \alpha, \tilde{p}),$ if the
following conditions 
\begin{equation*}
1+\frac{1}{\gamma}\left((1-\alpha)\frac{f(z)}{z}+\alpha f^{\prime
}(z)-1\right) \prec \widetilde{p(z)}=\dfrac{1+\tau^{2}z^{2}}{1-\tau z - \tau^{2}z^{2}},\;z\in\mathbb{D} 
\end{equation*}
and for $g(w)=f^{-1}(w)$ 
\begin{equation*}
1+\frac{1}{\gamma}\left((1-\alpha)\frac{g(w)}{w}+\alpha g^{\prime
}(w)-1\right)\prec \widetilde{p(w)}=\dfrac{1+\tau^{2}w^{2}}{1-\tau w - \tau^{2}w^{2}},\;w\in\mathbb{D} 
\end{equation*}
hold, where
$\tau = \dfrac{1-\sqrt{5}}{2} \approx -0.618.$

\item For $\lambda=0$ and $\alpha=1,$ we have the class $\mathcal{W}\mathcal{S}\mathcal{L}_{\Sigma}(\gamma, 0, 1, \tilde{p})\equiv \mathcal{H}\mathcal{S}\mathcal{L}_{\Sigma}(\gamma, \tilde{p}).$
A function $f\in \Sigma$ of the form 
\begin{equation*}
f(z) = z + \sum\limits_{n=2}^{\infty} a_{n} z^{n}, 
\end{equation*}
is said to be in $\mathcal{H}\mathcal{S}\mathcal{L}_{\Sigma}(\gamma, \tilde{p}),$ if the following
conditions 
\begin{equation*}
1+\frac{1}{\gamma}\left(f^{\prime }(z)-1\right) \prec \widetilde{p(z)}=\dfrac{1+\tau^{2}z^{2}}{1-\tau z - \tau^{2}z^{2}},\;z\in\mathbb{D} 
\end{equation*}
and for $g(w)=f^{-1}(w)$ 
\begin{equation*}
1+\frac{1}{\gamma}\left(g^{\prime }(w)-1\right) \prec \widetilde{p(w)}=\dfrac{1+\tau^{2}w^{2}}{1-\tau w - \tau^{2}w^{2}},\;w\in\mathbb{D} 
\end{equation*}
hold, where
$\tau = \dfrac{1-\sqrt{5}}{2} \approx -0.618.$
\end{enumerate}

\begin{definition}
\label{def1.2} A function $f\in \Sigma$ of the form 
\begin{equation*}
f(z) = z + \sum\limits_{n=2}^{\infty} a_{n} z^{n}, 
\end{equation*}
belongs to the class $\mathcal{R}\mathcal{S}\mathcal{L}_{\Sigma}(\gamma, \lambda, \tilde{p}),$ $\gamma\in \mathbb{C}\backslash\{0\}$ and $\lambda \geq 0,$  if the following conditions are satisfied: 
\begin{equation}  \label{C1-e1}
1+\frac{1}{\gamma}\left(\frac{z^{1-\lambda}f^{\prime}(z)}{(f(z))^{1-\lambda}}%
-1\right) \prec \widetilde{p(z)}=\dfrac{1+\tau^{2}z^{2}}{1-\tau z - \tau^{2}z^{2}},\;z\in\mathbb{D}
\end{equation}
and for $g(w)=f^{-1}(w)$ 
\begin{equation}  \label{C1-e2}
1+\frac{1}{\gamma}\left(\frac{w^{1-\lambda}g^{\prime}(w)}{(g(w))^{1-\lambda}}%
-1\right) \prec \widetilde{p(w)}=\dfrac{1+\tau^{2}w^{2}}{1-\tau w - \tau^{2}w^{2}},\;w\in\mathbb{D},
\end{equation}
where
$\tau = \dfrac{1-\sqrt{5}}{2} \approx -0.618.$
\end{definition}

\begin{enumerate}
\item For $\lambda=0,$ we have the class $\mathcal{R}\mathcal{S}\mathcal{L}_{\Sigma}(\gamma,
0, \tilde{p})\equiv \mathcal{S}\mathcal{L}_{\Sigma}(\gamma, \tilde{p}).$ A function $f\in
\Sigma$ of the form 
\begin{equation*}
f(z) = z + \sum\limits_{n=2}^{\infty} a_{n} z^{n}, 
\end{equation*}
is said to be in $\mathcal{S}\mathcal{L}_{\Sigma}(\gamma,\; \tilde{p}),$ if the following
conditions 
\begin{equation*}
1+\frac{1}{\gamma}\left(\frac{zf^{\prime }(z)}{f(z)}-1\right) \prec \widetilde{p(z)}=\dfrac{1+\tau^{2}z^{2}}{1-\tau z - \tau^{2}z^{2}},\;z\in\mathbb{D} 
\end{equation*}
and for $g(w)=f^{-1}(w)$ 
\begin{equation*}
1+\frac{1}{\gamma}\left(\frac{wg^{\prime }(w)}{g(w)}-1\right) \prec \widetilde{p(w)}=\dfrac{1+\tau^{2}w^{2}}{1-\tau w - \tau^{2}w^{2}},\;w\in\mathbb{D} 
\end{equation*}
hold, where
$\tau = \dfrac{1-\sqrt{5}}{2} \approx -0.618.$

\begin{remark}
For $\gamma=1$  the class $\mathcal{S}\mathcal{L}_{\Sigma}(1, \tilde{p})\equiv \mathcal{S}\mathcal{L}_{\Sigma}(\tilde{p})$ was introduced and studied  G\"{u}ney et al. \cite{HOG-GMS-JS-Fib-2018}.
\end{remark}

\item For $\lambda=1,$ we have the class $\mathcal{R}\mathcal{S}\mathcal{L}_{\Sigma}(\gamma,
1, \tilde{p})\equiv \mathcal{H}\mathcal{S}\mathcal{L}_{\Sigma}(\gamma, \tilde{p}).$ 
\end{enumerate}
\begin{definition}
	\label{def1.3} A function $f\in \Sigma$ of the form 
	\begin{equation*}
	f(z) = z + \sum\limits_{n=2}^{\infty} a_{n} z^{n}, 
	\end{equation*}
	belongs to the class $\mathcal{S}\mathcal{L}\mathcal{B}_{\Sigma}(\lambda; \tilde{p}),$  $\lambda \geq 1,$ if the following conditions are satisfied: 
	\begin{equation}  \label{C1-e1}
	\frac{z\left[f^{\prime}(z)\right]^{\lambda}}{f(z)} \prec \widetilde{p(z)}=\dfrac{1+\tau^{2}z^{2}}{1-\tau z - \tau^{2}z^{2}},\;z\in\mathbb{D}
	\end{equation}
	and for $g(w)=f^{-1}(w)$ 
	\begin{equation}  \label{C1-e2}
	\frac{w\left[g^{\prime}(w)\right]^{\lambda}}{g(w)} \prec \widetilde{p(w)}=\dfrac{1+\tau^{2}w^{2}}{1-\tau w - \tau^{2}w^{2}},\;w\in\mathbb{D},
	\end{equation}
	where
	$\tau = \dfrac{1-\sqrt{5}}{2} \approx -0.618.$
\end{definition}

\begin{enumerate}
	\item For $\lambda=1,$ we have the class $\mathcal{S}\mathcal{L}\mathcal{B}_{\Sigma}(1;\tilde{p})\equiv \mathcal{S}\mathcal{L}_{\Sigma}(\tilde{p}).$ A function $f\in
	\Sigma$ of the form 
	\begin{equation*}
	f(z) = z + \sum\limits_{n=2}^{\infty} a_{n} z^{n}, 
	\end{equation*}
	is said to be in $\mathcal{S}\mathcal{L}_{\Sigma}(\tilde{p}),$ if the following
	conditions 
	\begin{equation*}
	\frac{zf^{\prime}(z)}{f(z)} \prec \widetilde{p(z)}=\dfrac{1+\tau^{2}z^{2}}{1-\tau z - \tau^{2}z^{2}},\;z\in\mathbb{D}
	\end{equation*}
	and for $g(w)=f^{-1}(w)$ 
	\begin{equation*}
	\frac{wg^{\prime}(w)}{g(w)} \prec \widetilde{p(w)}=\dfrac{1+\tau^{2}w^{2}}{1-\tau w - \tau^{2}w^{2}},\;w\in\mathbb{D},
	\end{equation*}
	hold, where
	$\tau = \dfrac{1-\sqrt{5}}{2} \approx -0.618.$
\end{enumerate}
\begin{definition}
	\label{def1.4} A function $f\in \Sigma$ of the form 
	\begin{equation*}
	f(z) = z + \sum\limits_{n=2}^{\infty} a_{n} z^{n}, 
	\end{equation*}
	belongs to the class $\mathcal{P}\mathcal{S}\mathcal{L}_{\Sigma}(\lambda;\tilde{p}),$  $0 \leq \lambda \leq 1,$ if the following conditions are satisfied: 
	\begin{equation}  \label{C1-e1}
	\frac{zf^{\prime}(z)+\lambda z^{2}f^{\prime\prime}(z)}{(1-\lambda)f(z)+\lambda zf^{\prime}(z)} \prec \widetilde{p(z)}=\dfrac{1+\tau^{2}z^{2}}{1-\tau z - \tau^{2}z^{2}},\;z\in\mathbb{D}
	\end{equation}
	and for $g(w)=f^{-1}(w)$ 
	\begin{equation}  \label{C1-e2}
	\frac{wf^{\prime}(w)+\lambda w^{2}g^{\prime\prime}(w)}{(1-\lambda)g(w)+\lambda wg^{\prime}(w)} \prec \widetilde{p(w)}=\dfrac{1+\tau^{2}w^{2}}{1-\tau w - \tau^{2}w^{2}},\;w\in\mathbb{D},
	\end{equation}
	where
	$\tau = \dfrac{1-\sqrt{5}}{2} \approx -0.618.$
\end{definition}

\begin{enumerate}
	\item For $\lambda=0,$ we have the class $\mathcal{P}\mathcal{S}\mathcal{L}_{\Sigma}(0;\tilde{p})\equiv \mathcal{S}\mathcal{L}_{\Sigma}(\tilde{p}).$ A function $f\in
	\Sigma$ of the form 
	\begin{equation*}
	f(z) = z + \sum\limits_{n=2}^{\infty} a_{n} z^{n}, 
	\end{equation*}
	is said to be in $\mathcal{S}\mathcal{L}_{\Sigma}(\tilde{p}),$ if the following
	conditions 
	\begin{equation*}
	\frac{zf^{\prime}(z)}{f(z)} \prec \widetilde{p(z)}=\dfrac{1+\tau^{2}z^{2}}{1-\tau z - \tau^{2}z^{2}},\;z\in\mathbb{D}
	\end{equation*}
	and for $g(w)=f^{-1}(w)$ 
	\begin{equation*}
	\frac{wg^{\prime}(w)}{g(w)} \prec \widetilde{p(w)}=\dfrac{1+\tau^{2}w^{2}}{1-\tau w - \tau^{2}w^{2}},\;w\in\mathbb{D},
	\end{equation*}
	hold, where
	$\tau = \dfrac{1-\sqrt{5}}{2} \approx -0.618.$

	\item For $\lambda=1,$ we have the class $\mathcal{P}\mathcal{S}\mathcal{L}_{\Sigma}(1;\tilde{p})\equiv \mathcal{K}\mathcal{S}\mathcal{L}_{\Sigma}(\tilde{p}).$ A function $f\in
	\Sigma$ of the form 
	\begin{equation*}
	f(z) = z + \sum\limits_{n=2}^{\infty} a_{n} z^{n}, 
	\end{equation*}
	is said to be in $\mathcal{K}\mathcal{S}\mathcal{L}_{\Sigma}(\tilde{p}),$ if the following
	conditions 
	\begin{equation*}
	1 + \frac{z^{2}f^{\prime\prime}(z)}{f^{\prime}(z)} \prec \widetilde{p(z)}=\dfrac{1+\tau^{2}z^{2}}{1-\tau z - \tau^{2}z^{2}},\;z\in\mathbb{D}
	\end{equation*}
	and for $g(w)=f^{-1}(w)$ 
	\begin{equation*}
	1 + \frac{w^{2}g^{\prime\prime}(w)}{g^{\prime}(w)} \prec \widetilde{p(w)}=\dfrac{1+\tau^{2}w^{2}}{1-\tau w - \tau^{2}w^{2}},\;w\in\mathbb{D},
	\end{equation*}
	hold, where
	$\tau = \dfrac{1-\sqrt{5}}{2} \approx -0.618.$
	\begin{remark}
		For $\gamma=0,$ $\mathcal{P}\mathcal{S}\mathcal{L}_{\Sigma}(0, \tilde{p})\equiv \mathcal{S}\mathcal{L}_{\Sigma}(\tilde{p})$ and $\gamma=1,$  $\mathcal{P}\mathcal{S}\mathcal{L}_{\Sigma}(1, \tilde{p})\equiv \mathcal{K}\mathcal{S}\mathcal{L}_{\Sigma}(\tilde{p})$ the classes were introduced and studied  G\"{u}ney et al. \cite{HOG-GMS-JS-Fib-2018}.
	\end{remark}
\end{enumerate}

In order to prove our results for the function in the classes $\mathcal{W}\mathcal{S}\mathcal{L}
_{\Sigma}(\gamma, \lambda, \alpha, \tilde{p}),$  $\mathcal{R}\mathcal{S}\mathcal{L}_{\Sigma}(\gamma, \lambda, \tilde{p})$, $\mathcal{S}\mathcal{L}\mathcal{B}_{\Sigma}(\lambda;\tilde{p})$ and $\mathcal{P}\mathcal{S}\mathcal{L}_{\Sigma}(\lambda;\tilde{p})$, we need the following lemma. 

\begin{lemma}
	\cite{Pom}\label{lem-pom} If $p\in \mathcal{P},$ then $|p_i|\leqq 2$ for
	each $i,$ where $\mathcal{P}$ is the family of all functions $p,$ analytic
	in $\mathbb{D},$ for which
	\begin{equation*}
	\Re\{p(z)\}>0 \qquad (z \in \mathbb{D}),
	\end{equation*}
	where
	\begin{equation*}
	p(z)=1+p_1z+p_2z^2+\cdots \qquad (z \in \mathbb{D}).
	\end{equation*}
\end{lemma}
In this investigation, we find the estimates for the coefficients $|a_2|$
and $|a_3|$ for functions in the subclass $\mathcal{W}\mathcal{S}\mathcal{L}_{\Sigma}(\gamma,
\lambda, \alpha, \tilde{p})$,  $\mathcal{R}\mathcal{S}\mathcal{L}_{\Sigma}(\gamma,
\lambda, \tilde{p})$, $\mathcal{S}\mathcal{L}\mathcal{B}_{\Sigma}(\lambda;\tilde{p})$ and $\mathcal{P}\mathcal{S}\mathcal{L}_{\Sigma}(\lambda;\tilde{p})$ Also, we obtain the upper bounds using the results of $|a_2|$ and $|a_3|.$ 

\section{Initial Coefficient Estimates and Fekete-Szeg\"{o} Inequalities}

In the following theorem, we obtain coefficient estimates for functions in
the class $f\in\mathcal{W}\mathcal{S}\mathcal{L}_{\Sigma}(\gamma, \lambda, \alpha, \tilde{p}).$

\begin{theorem}
\label{Th1} Let $f(z)=z+\sum\limits_{n=2}^{\infty}a_{n}z^{n}$ be in the
class $\mathcal{W}\mathcal{S}\mathcal{L}_{\Sigma}(\gamma, \lambda, \alpha, \tilde{p})$. Then 
\begin{align*}
\left|a_{2}\right| &\leq \dfrac{\left|\gamma\right|\left|\tau\right|}{\sqrt{\gamma \tau\left(1+2\alpha + 2\lambda\right)+(1-3\tau)(1+\alpha)^{2}}},
\end{align*}
\begin{align*}
\left|a_3\right|&\leq \dfrac{\left|\gamma\right|\left|\tau\right|\left\{(1-3\tau)(1+\alpha)^{2}\right\}}{(1+2\alpha +2\lambda)\left[\gamma\tau(1+2\alpha +2\lambda)+(1-3\tau)(1+\alpha)^{2}\right]}
\end{align*}
and
\begin{eqnarray*}
\left|a_{3}-\mu a_{2}^{2}\right|
\leq \left\{
\begin{array}{ll}
\dfrac{\gamma\left|\tau\right|}{\left(1+2\alpha + 2\lambda\right)} &; 
0 \leq \left|h(\mu)\right|\leq \dfrac{\gamma\left|\tau\right|}{4\left(1+2\alpha + 2\lambda\right)}\\
4\left|h(\mu)\right| &; 
\left|h(\mu)\right|\geq \dfrac{\gamma\left|\tau\right|}{4\left(1+2\alpha + 2\lambda\right)},
\end{array}		
\right.
\end{eqnarray*}
where
\begin{eqnarray*}
	h(\mu) = \dfrac{\left(1-\mu\right)\gamma^{2}\tau^{2}}{4\left[\gamma\tau\left(1+2\alpha + 2\lambda\right)+\left(1+\alpha\right)^{2}\left(1-3\tau\right)\right]}.
\end{eqnarray*}
\end{theorem}

\begin{proof}
Since $f\in\mathcal{W}\mathcal{S}\mathcal{L}_{\Sigma}(\gamma, \lambda, \alpha, \tilde{p})$, from the
Definition \ref{def1.1} we have 
\begin{equation}  \label{2.2}
1+\frac{1}{\gamma}\left((1-\alpha+2\lambda)\frac{f(z)}{z}+(\alpha-2%
\lambda)f^{\prime }(z)+\lambda zf^{\prime \prime }(z)-1\right) = \widetilde{p(u(z))}
\end{equation}
and 
\begin{equation}  \label{2.3}
1+\frac{1}{\gamma}\left((1-\alpha+2\lambda)\frac{g(w)}{w}+(\alpha-2%
\lambda)g^{\prime }(w)+\lambda wg^{\prime \prime }(w)-1\right)=\widetilde{p(v(w))},
\end{equation}
where $z,\; w \in \mathbb{D}$ and $g=f^{-1}$. Using the fact the function $p(z)$ of the form  
\[
p(z) = 1+p_{1}z+p_{2}z^{2}+\ldots
\]
and 
$p\prec \tilde{p}.$ Then there exists an analytic function $u$ such that 
$\left|u(z)\right|<1$ in $\mathbb{D}$ and $p(z)=\tilde{p}(u(z)).$ Therefore,
define the function 
\begin{equation*}
h(z) = \dfrac{1+u(z)}{1-u(z)} = 1+c_{1}z+c_{2}z^{2}+\ldots
\end{equation*}
is in the class $\mathcal{P}.$ It follows that
\begin{equation*}
u(z) =\dfrac{h(z)-1}{h(z)+1} = \dfrac{c_{1}}{2}z + \left(c_{2}-\dfrac{c_{1}^{2}}{2}\right)\dfrac{z^{2}}{2}
+ \left(c_{3}-c_{1}c_{2}+\dfrac{c_{1}^{3}}{4}\right)\dfrac{z^{3}}{2}+\ldots
\end{equation*}
and 
\begin{eqnarray}
\tilde{p}(u(z)) &=& 1+\tilde{p}\left(\dfrac{c_{1}}{2}z + \left(c_{2}-\dfrac{c_{1}^{2}}{2}\right)\dfrac{z^{2}}{2}
+ \left(c_{3}-c_{1}c_{2}+\dfrac{c_{1}^{3}}{4}\right)\dfrac{z^{3}}{2}+\ldots\right)\nonumber\\
&& \quad + \tilde{p}_{2}\left(\dfrac{c_{1}}{2}z + \left(c_{2}-\dfrac{c_{1}^{2}}{2}\right)\dfrac{z^{2}}{2}
+ \left(c_{3}-c_{1}c_{2}+\dfrac{c_{1}^{3}}{4}\right)\dfrac{z^{3}}{2}+\ldots\right)^{2}\nonumber\\
&& \quad + \tilde{p}_{3} \left(\dfrac{c_{1}}{2}z + \left(c_{2}-\dfrac{c_{1}^{2}}{2}\right)\dfrac{z^{2}}{2}
+ \left(c_{3}-c_{1}c_{2}+\dfrac{c_{1}^{3}}{4}\right)\dfrac{z^{3}}{2}+\ldots\right)^{3} + \ldots \nonumber\\
&=& 1+ \dfrac{\tilde{p}_{1}c_{1}}{2}z 
+ \left(\dfrac{1}{2}\left(c_{2}-\dfrac{c_{1}^{2}}{2}\right)\tilde{p}_{1}
+\dfrac{c_{1}^{2}}{4}\tilde{p}_{2}\right)z^{2}\nonumber\\
&& \quad
+ \left(\dfrac{1}{2}\left(c_{3}-c_{1}c_{2}+\dfrac{c_{1}^{3}}{4}\right)\tilde{p}_{1}
+\dfrac{1}{2}c_{1}\left(c_{2}-\dfrac{c_{1}^{2}}{2}\right)\tilde{p}_{2}
+\dfrac{c_{1}^{3}}{8}\tilde{p}_{3}\right)z^{3} + \ldots.\label{2.4}
\end{eqnarray}
Similarly, there exists an analytic function $v$ such that $\left|v(w)\right|<1$ in $\mathbb{D}$ and $p(w)=\tilde{p}(v(w)).$ Therefore, the function
\begin{equation*}
k(w) = \dfrac{1+v(w)}{1-v(w)} = 1+d_{1}w+d_{2}w^{2}+\ldots
\end{equation*}
is in the class $\mathcal{P}.$ It follows that
\begin{equation*}
v(w) =\dfrac{k(w)-1}{k(w)+1} = \dfrac{d_{1}}{2}w + \left(d_{2}-\dfrac{d_{1}^{2}}{2}\right)\dfrac{w^{2}}{2}
+ \left(d_{3}-d_{1}d_{2}+\dfrac{d_{1}^{3}}{4}\right)\dfrac{w^{3}}{2}+\ldots
\end{equation*}
and 
\begin{eqnarray}
\tilde{p}(v(w)) &=& 1+\tilde{p}\left(\dfrac{d_{1}}{2}w + \left(d_{2}-\dfrac{d_{1}^{2}}{2}\right)\dfrac{w^{2}}{2}
+ \left(d_{3}-d_{1}d_{2}+\dfrac{d_{1}^{3}}{4}\right)\dfrac{w^{3}}{2}+\ldots\right)\nonumber\\
&& \quad + \tilde{p}_{2}\left(\dfrac{d_{1}}{2}w + \left(d_{2}-\dfrac{d_{1}^{2}}{2}\right)\dfrac{w^{2}}{2}
+ \left(d_{3}-d_{1}d_{2}+\dfrac{d_{1}^{3}}{4}\right)\dfrac{w^{3}}{2}+\ldots\right)^{2}\nonumber\\
&& \quad + \tilde{p}_{3} \left(\dfrac{d_{1}}{2}w + \left(d_{2}-\dfrac{d_{1}^{2}}{2}\right)\dfrac{w^{2}}{2}
+ \left(d_{3}-d_{1}d_{2}+\dfrac{d_{1}^{3}}{4}\right)\dfrac{w^{3}}{2}+\ldots\right)^{3} + \ldots \nonumber\\
&=& 1+ \dfrac{\tilde{p}_{1}d_{1}}{2}w 
+ \left(\dfrac{1}{2}\left(d_{2}-\dfrac{d_{1}^{2}}{2}\right)\tilde{p}_{1}
+\dfrac{d_{1}^{2}}{4}\tilde{p}_{2}\right)w^{2}\nonumber\\
&& \quad
+ \left(\dfrac{1}{2}\left(d_{3}-d_{1}d_{2}+\dfrac{d_{1}^{3}}{4}\right)\tilde{p}_{1}
+\dfrac{1}{2}d_{1}\left(d_{2}-\dfrac{d_{1}^{2}}{2}\right)\tilde{p}_{2}
+\dfrac{d_{1}^{3}}{8}\tilde{p}_{3}\right)w^{3} \label{2.5}\\
&& \quad + \ldots. \nonumber
\end{eqnarray}
By virtue of \eqref{2.2}, \eqref{2.3}, \eqref{2.4} and \eqref{2.5}, we have 
\begin{equation}  \label{2.6}
\frac{1}{\gamma}(1+\alpha)a_2=\dfrac{c_{1}\tau}{2},
\end{equation}
\begin{equation}  \label{2.7}
\frac{a_3}{\gamma}(1+2\alpha+2\lambda)=\dfrac{1}{2}\left(c_{2}-\dfrac{c_{1}^{2}}{2}\right)\tau+\dfrac{3c_{1}^{2}}{4}\tau^{2},
\end{equation}

\begin{equation}  \label{2.8}
-\frac{1}{\gamma}(1+\alpha)a_2=\dfrac{d_{1}\tau}{2},
\end{equation}

and 
\begin{equation}  \label{2.9}
\frac{(1+2\alpha+2\lambda)}{\gamma}(2a^2_2-a_3) =\dfrac{1}{2}\left(d_{2}-\dfrac{d_{1}^{2}}{2}\right)\tau+\dfrac{3d_{1}^{2}}{4}\tau^{2}.
\end{equation}

From \eqref{2.6} and \eqref{2.8}, we obtain
\begin{equation*} 
c_{1}=-d_{1},
\end{equation*}
and
\begin{eqnarray} 
\dfrac{2}{\gamma^{2}}(1+\alpha)^{2} a_{2}^2 &=& \dfrac{(c_{1}^{2} + d_{1}^{2})\tau^{2}}{4}\nonumber\\
a_{2}^2 &=& \dfrac{\gamma^{2}(c_{1}^{2} + d_{1}^{2})\tau^{2}}{8(1+\alpha)^{2}}~.\label{2.10}
\end{eqnarray}
By adding  \eqref{2.7} and \eqref{2.9}, we have 
\begin{eqnarray}
\dfrac{2}{\gamma}\left(1+2\alpha + 2\lambda\right) a_{2}^{2}
= \dfrac{1}{2}\left(c_{2}+d_{2}\right)\tau - \dfrac{1}{4}\left(c_{1}^{2}+d_{1}^{2}\right)\tau + \dfrac{3}{4}\left(c_{1}^{2}+d_{1}^{2}\right)\tau^{2}.\label{2.11}
\end{eqnarray}
By substituting  \eqref{2.10} in \eqref{2.11}, we reduce that
\begin{eqnarray}
a_{2}^{2} &=& \dfrac{\gamma^{2}\left(c_{2}+d_{2}\right)\tau^{2}}{4\left[\gamma \tau\left(1+2\alpha + 2\lambda\right)+(1-3\tau)(1+\alpha)^{2}\right]}.\label{2.12}
\end{eqnarray}
Now, applying Lemma \ref{lem-pom}, we obtain
\begin{eqnarray}
\left|a_{2}\right| &\leq& \dfrac{\left|\gamma\right|\left|\tau\right|}{\sqrt{\gamma \tau\left(1+2\alpha + 2\lambda\right)+(1-3\tau)(1+\alpha)^{2}}}.\label{2.13}
\end{eqnarray}
By subtracting \eqref{2.9} from \eqref{2.7}, we obtain
\begin{eqnarray}
a_{3}= \dfrac{\gamma\left(c_{2}-d_{2}\right)\tau}{4\left(1+2\alpha + 2\lambda\right)}  + a_{2}^{2}.\label{2.14}
\end{eqnarray}
Hence by Lemma \ref{lem-pom}, we have 
\begin{eqnarray}
\left|a_{3}\right|&\leq& \dfrac{\left|\gamma\right|\left(\left|c_{2}\right|+\left|d_{2}\right|\right)\left|\tau\right|}{4\left(1+2\alpha + 2\lambda\right)}  + \left|a_{2}\right|^{2}\leq  \dfrac{\left|\gamma\right|\left|\tau\right|}{\left(1+2\alpha + 2\lambda\right)} + \left|a_{2}\right|^{2}
.\label{2.15}
\end{eqnarray}
Then in view of \eqref{2.13}, we obtain
\begin{eqnarray*}
\left|a_{3}\right|
\leq 
\dfrac{\left|\gamma\right|\left|\tau\right|\left\{(1-3\tau)(1+\alpha)^{2}\right\}}{(1+2\alpha +2\lambda)\left[\gamma\tau(1+2\alpha +2\lambda)+(1-3\tau)(1+\alpha)^{2}\right]}
\end{eqnarray*}
	From \eqref{2.14}, we have 
	\begin{eqnarray}\label{Th1-Fekete-e1}
	a_{3} -\mu a_{2}^{2} = \dfrac{\gamma\left(c_{2}-d_{2}\right)\tau}{4\left(1+2\alpha + 2\lambda\right)} + \left(1-\mu\right)a_{2}^{2}.
	\end{eqnarray}
	By substituting \eqref{2.12} in \eqref{Th1-Fekete-e1}, we have 
	\begin{eqnarray}\label{Th1-Fekete-e2}
	a_{3} -\mu a_{2}^{2} 
	&=& \dfrac{\gamma\left(c_{2}-d_{2}\right)\tau}{4\left(1+2\alpha + 2\lambda\right)} + \left(1-\mu\right) \left(\dfrac{\gamma^{2}\left(c_{2}+d_{2}\right)\tau^{2}}{4\left[\gamma \tau\left(1+2\alpha + 2\lambda\right)+(1-3\tau)(1+\alpha)^{2}\right]}\right)\nonumber\\
	&=& \left(h(\mu) + \dfrac{\gamma\left|\tau\right|}{4\left(1+2\alpha + 2\lambda\right)}\right)c_{2}
		+ \left(h(\mu) - \dfrac{\gamma\left|\tau\right|}{4\left(1+2\alpha + 2\lambda\right)}\right)d_{2},
	\end{eqnarray}
	where
	\begin{eqnarray*}
	h(\mu) = \dfrac{\left(1-\mu\right)\gamma^{2}\tau^{2}}{4\left[\gamma\tau\left(1+2\alpha + 2\lambda\right)+\left(1+\alpha\right)^{2}\left(1-3\tau\right)\right]}.
	\end{eqnarray*}
	Thus by taking modulus of \eqref{Th1-Fekete-e2}, we conclude that
	\begin{eqnarray*}
	\left|a_{3}-\mu a_{2}^{2}\right|
	\leq \left\{
	\begin{array}{ll}
	\dfrac{\gamma\left|\tau\right|}{\left(1+2\alpha + 2\lambda\right)} &; 
	0 \leq \left|h(\mu)\right|\leq \dfrac{\gamma\left|\tau\right|}{4\left(1+2\alpha + 2\lambda\right)}\\
	4\left|h(\mu)\right| &; 
	\left|h(\mu)\right|\geq \dfrac{\gamma\left|\tau\right|}{4\left(1+2\alpha + 2\lambda\right)}.
	\end{array}		
	\right.
	\end{eqnarray*}
\end{proof}

\begin{theorem}
\label{thm3.1} Let $f(z)=z+\sum\limits_{n=2}^{\infty}a_{n}z^{n}$ be in the
class $\mathcal{R}\mathcal{S}\mathcal{L}_{\Sigma}(\gamma, \lambda, \tilde{p})$. Then 
\begin{align*}
\left|a_2\right|&\leq 
\dfrac{\sqrt{2}\left|\gamma\right|\left|\tau\right|}{\sqrt{\gamma \tau\left(2+\lambda\right)\left(1+\lambda\right)+2(1-3\tau)(1+\lambda)^{2}}},
\end{align*}
\begin{align*}
\left|a_3\right|&\leq \dfrac{\left|\gamma\right|\left|\tau\right|\left\{\gamma \tau \left(2+\lambda\right)\left(1+\lambda\right)+2(1-3\tau)(1+\lambda)^{2}-2\left(2+\lambda\right)\gamma\tau\right\}}{(2+\lambda)\left[\gamma\tau(2+\lambda)\left(1+\lambda\right)+2(1-3\tau)(1+\lambda)^{2}\right]}
\end{align*}
and
\begin{eqnarray*}
\left|a_{3}-\mu a_{2}^{2}\right|
\leq \left\{
\begin{array}{ll}
\dfrac{\left|\gamma\right|\left|\tau\right|}{2+\lambda} &; 
0 \leq \left|\mu-1\right|\leq \dfrac{M}{2\left|\gamma\right|\left|\tau\right|\left(2+\lambda\right)}\\
\dfrac{2\left|1-\mu\right|\gamma^{2}\tau^{2}}{M} &; 
\left|\mu-1\right|\geq \dfrac{M}{2\left|\gamma\right|\left|\tau\right|\left(2+\lambda\right)},
\end{array}		
\right.
\end{eqnarray*}
where
\begin{eqnarray*}
	M = \gamma\tau\left(2+\lambda\right)\left(1+\lambda\right)+2\left(1+\lambda\right)^{2}\left(1-3\tau\right).
\end{eqnarray*}
\end{theorem}

\begin{proof}
Since $f\in\mathcal{R}\mathcal{S}\mathcal{L}_{\Sigma}(\gamma, \lambda, \tilde{p})$, from the
Definition \ref{def1.2} we have 
\begin{equation}  \label{3.2}
1+\frac{1}{\gamma}\left(\frac{z^{1-\lambda}f^{\prime }(z)}{(f(z))^{1-\lambda}%
}-1\right) = \widetilde{p(u(z))}
\end{equation}
and 
\begin{equation}  \label{3.3}
1+\frac{1}{\gamma}\left(\frac{w^{1-\lambda}g^{\prime }(w)}{(g(w))^{1-\lambda}%
}-1\right) =\widetilde{p(v(w))}.
\end{equation}
By virtue of \eqref{3.2}, \eqref{3.3}, \eqref{2.4} and \eqref{2.5}, we get 
\begin{equation}  \label{3.6}
\frac{1}{\gamma}(1+\lambda)a_2=\dfrac{c_{1}\tau}{2},
\end{equation}
\begin{equation}  \label{3.7}
\frac{1}{\gamma}(2+\lambda)\left[a_{3}+
\left(\lambda-1\right)\dfrac{a_{2}^{2}}{2}\right]=\dfrac{1}{2}\left(c_{2}-\dfrac{c_{1}^{2}}{2}\right)\tau+\dfrac{3c_{1}^{2}}{4}\tau^{2},
\end{equation}
\begin{equation}  \label{3.8}
-\frac{1}{\gamma}(1+\lambda)a_2=\dfrac{d_{1}\tau}{2}
\end{equation}
and 
\begin{equation}  \label{3.9}
\frac{1}{\gamma}(2+\lambda)\left[\left(3+\lambda\right)\dfrac{a_{2}^{2}}{2}-a_{3}\right] =\dfrac{1}{2}\left(d_{2}-\dfrac{d_{1}^{2}}{2}\right)\tau+\dfrac{3d_{1}^{2}}{4}\tau^{2}.
\end{equation}
From \eqref{3.6} and \eqref{3.8}, we obtain
\begin{equation*} 
c_{1}=-d_{1},
\end{equation*}
and
\begin{eqnarray} 
\dfrac{2}{\gamma^{2}}(1+\lambda)^{2} a_{2}^2 &=& \dfrac{(c_{1}^{2} + d_{1}^{2})\tau^{2}}{4}\nonumber\\
a_{2}^2 &=& \dfrac{\gamma^{2}(c_{1}^{2} + d_{1}^{2})\tau^{2}}{8(1+\lambda)^{2}}~.\label{3.10}
\end{eqnarray}
By adding  \eqref{3.7} and \eqref{3.9}, we have 
\begin{eqnarray}
\dfrac{1}{\gamma}\left(2+\lambda\right)\left(1+\lambda\right) a_{2}^{2}
= \dfrac{1}{2}\left(c_{2}+d_{2}\right)\tau - \dfrac{1}{4}\left(c_{1}^{2}+d_{1}^{2}\right)\tau + \dfrac{3}{4}\left(c_{1}^{2}+d_{1}^{2}\right)\tau^{2}.\label{3.11}
\end{eqnarray}
By substituting  \eqref{3.10} in \eqref{3.11}, we reduce that
\begin{eqnarray}
a_{2}^{2} &=& \dfrac{\gamma^{2}\left(c_{2}+d_{2}\right)\tau^{2}}{2\left[\gamma \tau\left(2+\lambda\right)\left(1+\lambda\right)+2(1-3\tau)(1+\lambda)^{2}\right]}.\label{3.12}
\end{eqnarray}
Now, applying Lemma \ref{lem-pom}, we obtain
\begin{eqnarray}
\left|a_{2}\right| &\leq& \dfrac{\sqrt{2}\left|\gamma\right|\left|\tau\right|}{\sqrt{\gamma \tau\left(2+\lambda\right)\left(1+\lambda\right)+2(1-3\tau)(1+\lambda)^{2}}}.\label{3.13}
\end{eqnarray}
By subtracting \eqref{3.9} from \eqref{3.7}, we obtain
\begin{eqnarray}
a_{3}= \dfrac{\gamma\left(c_{2}-d_{2}\right)\tau}{4\left(2+\lambda\right)}  + a_{2}^{2}.\label{3.14}
\end{eqnarray}
Hence by Lemma \ref{lem-pom}, we have 
\begin{eqnarray}
\left|a_{3}\right|&\leq& \dfrac{\left|\gamma\right|\left(\left|c_{2}\right|+\left|d_{2}\right|\right)\left|\tau\right|}{4\left(2+\lambda\right)}  + \left|a_{2}\right|^{2}\leq  \dfrac{\left|\gamma\right|\left|\tau\right|}{\left(2+\lambda\right)} + \left|a_{2}\right|^{2}
.\label{3.15}
\end{eqnarray}
Then in view of \eqref{3.13}, we obtain
\begin{eqnarray*}
	\left|a_{3}\right|
	\leq 
	\dfrac{\left|\gamma\right|\left|\tau\right|\left\{\gamma \tau \left(2+\lambda\right)\left(1+\lambda\right)+2(1-3\tau)(1+\lambda)^{2}-2\left(2+\lambda\right)\gamma\tau\right\}}{(2+\lambda)\left[\gamma\tau(2+\lambda)\left(1+\lambda\right)+2(1-3\tau)(1+\lambda)^{2}\right]}.
\end{eqnarray*}
	From \eqref{3.14}, we have 
	\begin{eqnarray}\label{Th2-Fekete-e1}
	a_{3} -\mu a_{2}^{2} = \dfrac{\gamma\left(c_{2}-d_{2}\right)\tau}{4\left(2+\lambda\right)} + \left(1-\mu\right)a_{2}^{2}.
	\end{eqnarray}
	By substituting \eqref{3.12} in \eqref{Th2-Fekete-e1}, we have 
	\begin{eqnarray}\label{Th2-Fekete-e2}
	a_{3} -\mu a_{2}^{2} 
	&=& \dfrac{\gamma\left(c_{2}-d_{2}\right)\tau}{4\left(2+\lambda\right)} + \left(1-\mu\right) \left(\dfrac{\gamma^{2}\left(c_{2}+d_{2}\right)\tau^{2}}{2\left[\gamma\tau\left(2+\lambda\right)\left(1+\lambda\right)+2\left(1+\lambda\right)^{2}\left(1-3\tau\right)\right]}\right)\nonumber\\
	&=& \left(h(\mu) + \dfrac{\left|\gamma\right|\left|\tau\right|}{4\left(2+\lambda\right)}\right)c_{2}
	+ \left(h(\mu) - \dfrac{\left|\gamma\right|\left|\tau\right|}{4\left(2+\lambda\right)}\right)d_{2},
	\end{eqnarray}
	where
	\begin{eqnarray*}
		h(\mu) = \dfrac{\left(1-\mu\right)\gamma^{2}\tau^{2}}{2\left[\gamma\tau\left(2+\lambda\right)\left(1+\lambda\right)+2\left(1+\lambda\right)^{2}\left(1-3\tau\right)\right]}.
	\end{eqnarray*}
	Thus by taking modulus of \eqref{Th2-Fekete-e2}, we conclude that
	\begin{eqnarray*}
		\left|a_{3}-\mu a_{2}^{2}\right|
		\leq \left\{
		\begin{array}{ll}
			\dfrac{\gamma\left|\tau\right|}{\left(2+\lambda\right)} &; 
			0 \leq \left|h(\mu)\right|\leq \dfrac{\gamma\left|\tau\right|}{4\left(2+\lambda\right)}\\
			4\left|h(\mu)\right| &; 
			\left|h(\mu)\right|\geq \dfrac{\gamma\left|\tau\right|}{4\left(2+\lambda\right)}.
		\end{array}		
		\right.
	\end{eqnarray*}
This gives desired inequality. 
\end{proof}
\begin{theorem}
	\label{thm4.1} Let $f(z)=z+\sum\limits_{n=2}^{\infty}a_{n}z^{n}$ be in the
	class $\mathcal{S}\mathcal{L}\mathcal{B}_{\Sigma}(\lambda;\tilde{p})$. Then 
	\begin{align*}
	\left|a_2\right|&\leq \dfrac{\left|\tau\right|}{\sqrt{\left(2\lambda-1\right)\left[\tau\left(3-5\lambda\right)+2\lambda-1\right]}},
	\end{align*}
	\begin{align*}
	\left|a_3\right|&\leq \dfrac{\left|\tau\right|\left[(2\lambda-1)^{2}-2\left(5\lambda^{2}-4\lambda+1\right)\tau\right\}}{(2\lambda-1)(3\lambda-1)[(3-5\lambda)\tau+2\lambda-1]}
	\end{align*}
	and
	\begin{eqnarray*}
	\left|a_{3}-\mu a_{2}^{2}\right|
	\leq \left\{
	\begin{array}{ll}
	\dfrac{\left|\gamma\right|\left|\tau\right|}{3\lambda-1} &; 
	0 \leq \left|\mu-1\right|\leq \dfrac{M}{\left|\tau\right|\left(3\lambda-1\right)}\\
	\dfrac{\left|1-\mu\right|\tau^{2}}{M} &; 
	\left|\mu-1\right|\geq \dfrac{M}{\left|\tau\right|\left(3\lambda-1\right)},
	\end{array}		
	\right.
	\end{eqnarray*}
	where
	\begin{eqnarray*}
		M = \left(2\lambda-1\right)\left[\tau\left(3-5\lambda\right)+2\lambda-1\right].
	\end{eqnarray*}
\end{theorem}

\begin{proof}
	Since $f\in\mathcal{S}\mathcal{L}\mathcal{B}_{\Sigma}(\lambda;\tilde{p})$, from the
	Definition \ref{def1.3} we have 
	\begin{equation}  \label{4.2}
	\frac{z\left[f^{\prime}(z)\right]^{\lambda}}{f(z)}
	 = \widetilde{p(u(z))}
	\end{equation}
	and 
	\begin{equation}  \label{4.3}
	\frac{w\left[g^{\prime}(w)\right]^{\lambda}}{g(w)}
	 =\widetilde{p(v(w))}.
	\end{equation}
	By virtue of \eqref{4.2}, \eqref{4.3}, \eqref{2.4} and \eqref{2.5}, we get 
	\begin{equation}  \label{4.6}
	(2\lambda-1)a_2=\dfrac{c_{1}\tau}{2},
	\end{equation}
	\begin{equation}  \label{4.7}
	(3\lambda-1)a_{3}+(2\lambda^{2}-4\lambda+1)a_{2}^{2}=\dfrac{1}{2}\left(c_{2}-\dfrac{c_{1}^{2}}{2}\right)\tau+\dfrac{3c_{1}^{2}}{4}\tau^{2},
	\end{equation}
	
	\begin{equation}  \label{4.8}
	-(2\lambda-1)a_2=\dfrac{d_{1}\tau}{2}
	\end{equation}
	and 
	\begin{equation}  \label{4.9}
	(2\lambda^{2}+2\lambda-1)a_{2}^{2} - (3\lambda-1)a_{3} =\dfrac{1}{2}\left(d_{2}-\dfrac{d_{1}^{2}}{2}\right)\tau+\dfrac{3d_{1}^{2}}{4}\tau^{2}.
	\end{equation}
	
	From \eqref{4.6} and \eqref{4.8}, we obtain
	\begin{equation*} 
	c_{1}=-d_{1},
	\end{equation*}
	and
	\begin{eqnarray} 
	2(2\lambda-1)^{2} a_{2}^{2} &=& \dfrac{(c_{1}^{2} + d_{1}^{2})\tau^{2}}{4}\nonumber\\
	a_{2}^{2} &=& \dfrac{(c_{1}^{2} + d_{1}^{2})\tau^{2}}{8(2\lambda-1)^{2}}~.\label{4.10}
	\end{eqnarray}
	By adding  \eqref{4.7} and \eqref{4.9}, we have 
	\begin{eqnarray}
	2\lambda\left(2\lambda-1\right) a_{2}^{2}
	= \dfrac{1}{2}\left(c_{2}+d_{2}\right)\tau - \dfrac{1}{4}\left(c_{1}^{2}+d_{1}^{2}\right)\tau + \dfrac{3}{4}\left(c_{1}^{2}+d_{1}^{2}\right)\tau^{2}.\label{4.11}
	\end{eqnarray}
	By substituting  \eqref{4.10} in \eqref{4.11}, we reduce that
	\begin{eqnarray}
	a_{2}^{2} &=& \dfrac{\left(c_{2}+d_{2}\right)\tau^{2}}{4\left(2\lambda-1\right)\left[\tau\left(3-5\lambda\right)+2\lambda-1\right]}.\label{4.12}
	\end{eqnarray}
	Now, applying Lemma \ref{lem-pom}, we obtain
	\begin{eqnarray}
	\left|a_{2}\right| &\leq& \dfrac{\left|\tau\right|}{\sqrt{\left(2\lambda-1\right)\left[\tau\left(3-5\lambda\right)+2\lambda-1\right]}}.\label{4.13}
	\end{eqnarray}
	By subtracting \eqref{4.9} from \eqref{4.7}, we obtain
	\begin{eqnarray}
	a_{3}= \dfrac{\left(c_{2}-d_{2}\right)\tau}{4\left(3\lambda-1\right)}  + a_{2}^{2}.\label{4.14}
	\end{eqnarray}
	Hence by Lemma \ref{lem-pom}, we have 
	\begin{eqnarray}
	\left|a_{3}\right|&\leq& \dfrac{\left(\left|c_{2}\right|+\left|d_{2}\right|\right)\left|\tau\right|}{4\left(3\lambda-1\right)}  + \left|a_{2}\right|^{2}\leq  \dfrac{\left|\tau\right|}{3\lambda-1} + \left|a_{2}\right|^{2}
	.\label{4.15}
	\end{eqnarray}
	Then in view of \eqref{4.13}, we obtain
	\begin{eqnarray*}
		\left|a_{3}\right|
		\leq 
		\dfrac{\left|\tau\right|\left[(2\lambda-1)^{2}-2\left(5\lambda^{2}-4\lambda+1\right)\tau\right\}}{(2\lambda-1)(3\lambda-1)[(3-5\lambda)\tau+2\lambda-1]}
	\end{eqnarray*}
	From \eqref{4.14}, we have 
	\begin{eqnarray}\label{Th3-Fekete-e1}
	a_{3} -\mu a_{2}^{2} = \dfrac{\left(c_{2}-d_{2}\right)\tau}{4\left(3\lambda-1\right)} + \left(1-\mu\right)a_{2}^{2}.
	\end{eqnarray}
	By substituting \eqref{4.12} in \eqref{Th3-Fekete-e1}, we have 
	\begin{eqnarray}\label{Th3-Fekete-e2}
	a_{3} -\mu a_{2}^{2} 
	&=& \dfrac{\left(c_{2}-d_{2}\right)\tau}{4\left(3\lambda-1\right)} + \left(1-\mu\right) \left(\dfrac{\left(c_{2}+d_{2}\right)\tau^{2}}{4\left(2\lambda-1\right)\left[\tau\left(3-5\lambda\right)+2\lambda-1\right]}\right)\nonumber\\
	&=& \left(h(\mu) + \dfrac{\left|\tau\right|}{4\left(3\lambda-1\right)}\right)c_{2}
	+ \left(h(\mu) - \dfrac{\left|\tau\right|}{4\left(3\lambda-1\right)}\right)d_{2},
	\end{eqnarray}
	where
	\begin{eqnarray*}
		h(\mu) = \dfrac{\left(1-\mu\right)\tau^{2}}{4\left(2\lambda-1\right)\left[\tau\left(3-5\lambda\right)+2\lambda-1\right]}.
	\end{eqnarray*}
	Thus by taking modulus of \eqref{Th3-Fekete-e2}, we conclude that
	\begin{eqnarray*}
		\left|a_{3}-\mu a_{2}^{2}\right|
		\leq \left\{
		\begin{array}{ll}
			\dfrac{\left|\tau\right|}{\left(3\lambda-1\right)} &; 
			0 \leq \left|h(\mu)\right|\leq \dfrac{\left|\tau\right|}{4\left(3\lambda-1\right)}\\
			4\left|h(\mu)\right| &; 
			\left|h(\mu)\right|\geq \dfrac{\left|\tau\right|}{4\left(3\lambda-1\right)}.
		\end{array}		
		\right.
	\end{eqnarray*}
	This gives desired inequality. 
\end{proof}
\begin{theorem}
	\label{thm5.1} Let $f(z)=z+\sum\limits_{n=2}^{\infty}a_{n}z^{n}$ be in the
	class $\mathcal{P}\mathcal{S}\mathcal{L}_{\Sigma}(\lambda;\tilde{p})$. Then 
	\begin{align*}
	\left|a_2\right|&\leq \dfrac{\left|\tau\right|}{\sqrt{\left(1+\lambda\right)^{2}-2\tau\left(2\lambda^{2}+2\lambda+1\right)}},
	\end{align*}
	\begin{align*}
	\left|a_3\right|&\leq 
	\dfrac{\left|\tau\right|\left(1-4\tau\right)(1+\lambda)^{2}}{2(1+2\lambda)\left[\left(1+\lambda\right)^{2}-2\tau\left(2\lambda^{2}+2\lambda+1\right)\right]}
	\end{align*}
	and
	\begin{eqnarray*}
	\left|a_{3}-\mu a_{2}^{2}\right|
	\leq \left\{
	\begin{array}{ll}
	\dfrac{\left|\tau\right|}{2+4\lambda} &; 
	0 \leq \left|\mu-1\right|\leq \dfrac{M}{2\left|\tau\right|(1+2\lambda)}\\
	\dfrac{\left|1-\mu\right|\tau^{2}}{M} &; 
	\left|\mu-1\right|\geq \dfrac{M}{2\left|\tau\right|(1+2\lambda)},
	\end{array}		
	\right.
	\end{eqnarray*}
	where
	\begin{eqnarray*}
		M = \left(1+\lambda\right)^{2}-2\tau\left(2\lambda^{2}+2\lambda+1\right).
	\end{eqnarray*}
\end{theorem}

\begin{proof}
	Since $f\in\mathcal{P}\mathcal{S}\mathcal{L}_{\Sigma}(\lambda;\tilde{p})$, from the
	Definition \ref{def1.4} we have 
	\begin{equation}  \label{5.2}
	\frac{zf^{\prime}(z)+\lambda z^{2}f^{\prime\prime}(z)}{(1-\lambda)f(z)+\lambda zf^{\prime}(z)}
	= \widetilde{p(u(z))}
	\end{equation}
	and 
	\begin{equation}  \label{5.3}
	\frac{wf^{\prime}(w)+\lambda w^{2}g^{\prime\prime}(w)}{(1-\lambda)g(w)+\lambda wg^{\prime}(w)}
	=\widetilde{p(v(w))}.
	\end{equation}
	By virtue of \eqref{5.2}, \eqref{5.3}, \eqref{2.4} and \eqref{2.5}, we get 
	\begin{equation}  \label{5.6}
	(1+\lambda)a_2=\dfrac{c_{1}\tau}{2},
	\end{equation}
	\begin{equation}  \label{5.7}
	2(1+2\lambda)a_{3}-(1+\lambda)^{2}a_{2}^{2}=\dfrac{1}{2}\left(c_{2}-\dfrac{c_{1}^{2}}{2}\right)\tau+\dfrac{3c_{1}^{2}}{4}\tau^{2},
	\end{equation}
	
	\begin{equation}  \label{5.8}
	-(1+\lambda)a_2=\dfrac{d_{1}\tau}{2}
	\end{equation}
	and 
	\begin{equation}  \label{5.9}
	-2(1+2\lambda)a_{3}-(\lambda^{2}-6\lambda-3)a_{2}^{2}  =\dfrac{1}{2}\left(d_{2}-\dfrac{d_{1}^{2}}{2}\right)\tau+\dfrac{3d_{1}^{2}}{4}\tau^{2}.
	\end{equation}
	
	From \eqref{5.6} and \eqref{5.8}, we obtain
	\begin{equation*} 
	c_{1}=-d_{1},
	\end{equation*}
	and
	\begin{eqnarray} 
	2(1+\lambda)^{2} a_{2}^{2} &=& \dfrac{(c_{1}^{2} + d_{1}^{2})\tau^{2}}{4}\nonumber\\
	a_{2}^{2} &=& \dfrac{(c_{1}^{2} + d_{1}^{2})\tau^{2}}{8(1+\lambda)^{2}}~.\label{5.10}
	\end{eqnarray}
	By adding  \eqref{5.7} and \eqref{5.9}, we have 
	\begin{eqnarray}
	2\left(1+2\lambda-\lambda^{2}\right) a_{2}^{2}
	= \dfrac{1}{2}\left(c_{2}+d_{2}\right)\tau - \dfrac{1}{4}\left(c_{1}^{2}+d_{1}^{2}\right)\tau + \dfrac{3}{4}\left(c_{1}^{2}+d_{1}^{2}\right)\tau^{2}.\label{5.11}
	\end{eqnarray}
	By substituting  \eqref{5.10} in \eqref{5.11}, we reduce that
	\begin{eqnarray}
	a_{2}^{2} &=& \dfrac{\left(c_{2}+d_{2}\right)\tau^{2}}{4\left[\left(1+\lambda\right)^{2}-2\tau\left(2\lambda^{2}+2\lambda+1\right)\right]}.\label{5.12}
	\end{eqnarray}
	Now, applying Lemma \ref{lem-pom}, we obtain
	\begin{eqnarray}
	\left|a_{2}\right| &\leq& \dfrac{\left|\tau\right|}{\sqrt{\left(1+\lambda\right)^{2}-2\tau\left(2\lambda^{2}+2\lambda+1\right)}}.\label{5.13}
	\end{eqnarray}
	By subtracting \eqref{5.9} from \eqref{5.7}, we obtain
	\begin{eqnarray}
	a_{3}= \dfrac{\left(c_{2}-d_{2}\right)\tau}{8\left(1+2\lambda\right)}  + a_{2}^{2}.\label{5.14}
	\end{eqnarray}
	Hence by Lemma \ref{lem-pom}, we have 
	\begin{eqnarray*}
	\left|a_{3}\right|&\leq& \dfrac{\left(\left|c_{2}\right|+\left|d_{2}\right|\right)\left|\tau\right|}{8\left(1+2\lambda\right)}  + \left|a_{2}\right|^{2}\leq  \dfrac{\left|\tau\right|}{2+4\lambda} + \left|a_{2}\right|^{2}
	.\label{5.15}
	\end{eqnarray*}
	Then in view of \eqref{5.13}, we obtain
	\begin{eqnarray*}
		\left|a_{3}\right|
		\leq 
		\dfrac{\left|\tau\right|\left(1-4\tau\right)(1+\lambda)^{2}}{2(1+2\lambda)\left[\left(1+\lambda\right)^{2}-2\tau\left(2\lambda^{2}+2\lambda+1\right)\right]}
	\end{eqnarray*}
	From \eqref{5.14}, we have 
	\begin{eqnarray}\label{Th4-Fekete-e1}
	a_{3} -\mu a_{2}^{2} = \dfrac{\left(c_{2}-d_{2}\right)\tau}{8\left(1+2\lambda\right)} + \left(1-\mu\right)a_{2}^{2}.
	\end{eqnarray}
	By substituting \eqref{5.12} in \eqref{Th4-Fekete-e1}, we have 
	\begin{eqnarray}\label{Th4-Fekete-e2}
	a_{3} -\mu a_{2}^{2} 
	&=& \dfrac{\left(c_{2}-d_{2}\right)\tau}{8\left(1+2\lambda\right)} + \left(1-\mu\right) \left(\dfrac{\left(c_{2}+d_{2}\right)\tau^{2}}{4\left[\left(1+\lambda\right)^{2}-2\tau\left(2\lambda^{2}+2\lambda+1\right)\right]}\right)\nonumber\\
	&=& \left(h(\mu) + \dfrac{\tau}{8\left(1+2\lambda\right)}\right)c_{2}
	+ \left(h(\mu) - \dfrac{\tau}{8\left(1+2\lambda\right)}\right)d_{2},
	\end{eqnarray}
	where
	\begin{eqnarray*}
		h(\mu) = \dfrac{\left(1-\mu\right)\tau^{2}}{4\left[\left(1+\lambda\right)^{2}-2\tau\left(2\lambda^{2}+2\lambda+1\right)\right]}.
	\end{eqnarray*}
	Thus by taking modulus of \eqref{Th4-Fekete-e2}, we conclude that
	\begin{eqnarray*}
		\left|a_{3}-\mu a_{2}^{2}\right|
		\leq \left\{
		\begin{array}{ll}
			\dfrac{\left|\tau\right|}{2\left(1+2\lambda\right)} &; 
			0 \leq \left|h(\mu)\right|\leq \dfrac{\left|\tau\right|}{8\left(1+2\lambda\right)}\\
			4\left|h(\mu)\right| &; 
			\left|h(\mu)\right|\geq \dfrac{\left|\tau\right|}{8\left(1+2\lambda\right)}.
		\end{array}		
		\right.
	\end{eqnarray*}
	This gives desired inequality. 
\end{proof}
\section{Corollaries and Consequences}
\begin{corollary}
	\label{cor2.1} Let $f(z)=z+\sum\limits_{n=2}^{\infty}a_{n}z^{n}$ be in the
	class $\mathcal{F}\mathcal{S}\mathcal{L}_{\Sigma}(\gamma, \lambda, \tilde{p})$. Then 
	\begin{align*}
	\left|a_{2}\right| &\leq \dfrac{\left|\gamma\right|\left|\tau\right|}{\sqrt{3\gamma \tau\left(1+ 2\lambda\right)+4(1-3\tau)(1+\lambda)^{2}}},\qquad
	\left|a_3\right|\leq \dfrac{4\left|\gamma\right|\left|\tau\right|(1-3\tau)(1+\alpha)^{2}}{3(1+2\lambda)\left[3\gamma\tau(1+2\lambda)+4(1-3\tau)(1+\lambda)^{2}\right]}
	\end{align*}
	and
	\begin{eqnarray*}
	\left|a_{3}-\mu a_{2}^{2}\right|
	\leq \left\{
	\begin{array}{ll}
	\dfrac{\left|\gamma\right|\left|\tau\right|}{3+6\lambda} &; 
	0 \leq \left|h(\mu)\right|\leq \dfrac{\left|\gamma\right|\left|\tau\right|}{12+24\lambda}\\
	4\left|h(\mu)\right| &; 
	\left|h(\mu)\right|\geq \dfrac{\left|\gamma\right|\left|\tau\right|}{12+24\lambda},
	\end{array}		
	\right.
	\end{eqnarray*}
	where
	\begin{eqnarray*}
		h(\mu) = \dfrac{\left(1-\mu\right)\gamma^{2}\tau^{2}}{4\left[3\gamma\tau(1+2\lambda)+4(1-3\tau)(1+\lambda)^{2}\right]}.
	\end{eqnarray*}
\end{corollary}

\begin{corollary}
	\label{cor2.2} Let $f(z)=z+\sum\limits_{n=2}^{\infty}a_{n}z^{n}$ be in the
	class $\mathcal{B}\mathcal{S}\mathcal{L}_{\Sigma}(\gamma, \alpha, \tilde{p})$. Then 
	\begin{align*}
	\left|a_{2}\right| &\leq \dfrac{\left|\gamma\right|\left|\tau\right|}{\sqrt{\gamma \tau\left(1+2\alpha\right)+(1-3\tau)(1+\alpha)^{2}}},\qquad
	\left|a_3\right|\leq \dfrac{\left|\gamma\right|\left|\tau\right|\left\{(1-3\tau)(1+\alpha)^{2}\right\}}{(1+2\alpha)\left[\gamma\tau(1+2\alpha)+(1-3\tau)(1+\alpha)^{2}\right]}
	\end{align*}
	and
	\begin{eqnarray*}
		\left|a_{3}-\mu a_{2}^{2}\right|
		\leq \left\{
		\begin{array}{ll}
			\dfrac{\left|\gamma\right|\left|\tau\right|}{1+2\alpha} &; 
			0 \leq \left|h(\mu)\right|\leq \dfrac{\left|\gamma\right|\left|\tau\right|}{4+8\alpha}\\
			4\left|h(\mu)\right| &; 
			\left|h(\mu)\right|\geq \dfrac{\left|\gamma\right|\left|\tau\right|}{4+8\alpha},
		\end{array}		
		\right.
	\end{eqnarray*}
	where
	\begin{eqnarray*}
		h(\mu) = \dfrac{\left(1-\mu\right)\gamma^{2}\tau^{2}}{4\left[\gamma\tau\left(1+2\alpha\right)+\left(1+\alpha\right)^{2}\left(1-3\tau\right)\right]}.
	\end{eqnarray*}
\end{corollary}

\begin{corollary}
	\label{cor2.3} Let $f(z)=z+\sum\limits_{n=2}^{\infty}a_{n}z^{n}$ be in the
	class $\mathcal{H}\mathcal{S}\mathcal{L}_{\Sigma}(\gamma, \tilde{p})$. Then 
	\begin{align*}
	\left|a_{2}\right| &\leq \dfrac{\left|\gamma\right|\left|\tau\right|}{\sqrt{3\gamma \tau+4(1-3\tau)}},
	\qquad
	\left|a_3\right|\leq \dfrac{4\left|\gamma\right|\left|\tau\right|(1-3\tau)}{3\left[3\gamma\tau+4(1-3\tau)\right]},
	\end{align*}
	and
	\begin{eqnarray*}
		\left|a_{3}-\mu a_{2}^{2}\right|
		\leq \left\{
		\begin{array}{ll}
			\dfrac{\left|\gamma\right|\left|\tau\right|}{3} &; 
			0 \leq \left|h(\mu)\right|\leq \dfrac{\left|\gamma\right|\left|\tau\right|}{12}\\
			4\left|h(\mu)\right| &; 
			\left|h(\mu)\right|\geq \dfrac{\left|\gamma\right|\left|\tau\right|}{12},
		\end{array}		
		\right.
	\end{eqnarray*}
	where
	\begin{eqnarray*}
		h(\mu) = \dfrac{\left(1-\mu\right)\gamma^{2}\tau^{2}}{4\left[3\gamma\tau+4\left(1-3\tau\right)\right]}.
	\end{eqnarray*}
\end{corollary}
\begin{corollary}\cite{HOG-GMS-JS-Fib-2018}
	\label{cor4.1} Let $f(z)=z+\sum\limits_{n=2}^{\infty}a_{n}z^{n}$ be in the
	class $\mathcal{S}\mathcal{L}_{\Sigma}(\tilde{p}).$ Then 
	\begin{align*}
	\left|a_2\right|&\leq \dfrac{\left|\tau\right|}{\sqrt{1-2\tau}}, \;\;\;\; 
	\left|a_3\right| \leq \dfrac{\left|\tau\right|(1-4\tau)}{2-4\tau}
	\end{align*}
	and
	\begin{eqnarray*}
	\left|a_{3}-\mu a_{2}^{2}\right|
	\leq \left\{
	\begin{array}{ll}
	\dfrac{\left|\tau\right|}{2} &; 
	0 \leq \left|\mu-1\right|\leq \dfrac{1-2\tau}{2\left|\tau\right|}\\
	\dfrac{\left|1-\mu\right|\tau^{2}}{1-2\tau} &; 
	\left|\mu-1\right|\geq \dfrac{1-2\tau}{2\left|\tau\right|}.
	\end{array}		
	\right.
	\end{eqnarray*}
\end{corollary}
\begin{corollary}\cite{HOG-GMS-JS-Fib-2018}
	\label{cor4.2} Let $f(z)=z+\sum\limits_{n=2}^{\infty}a_{n}z^{n}$ be in the
	class $\mathcal{K}\mathcal{S}\mathcal{L}_{\Sigma}(\tilde{p}).$ Then 
	\begin{align*}
	\left|a_2\right|&\leq \dfrac{\left|\tau\right|}{\sqrt{4-10\tau}}, \;\;\;\; 
	\left|a_3\right| \leq \dfrac{\left|\tau\right|(1-4\tau)}{6-15\tau}.
	\end{align*}
	and
	\begin{eqnarray*}
	\left|a_{3}-\mu a_{2}^{2}\right|
	\leq \left\{
	\begin{array}{ll}
	\dfrac{\left|\tau\right|}{6} &; 
	0 \leq \left|\mu-1\right|\leq \dfrac{2-5\tau}{3\left|\tau\right|}\\
	\dfrac{\left|1-\mu\right|\tau^{2}}{4-10\tau} &; 
	\left|\mu-1\right|\geq \dfrac{2-5\tau}{6\left|\tau\right|}.
	\end{array}		
	\right.
	\end{eqnarray*}
\end{corollary}
\begin{remark}
	Results discussed in Corollaries \ref{cor4.1} and \ref{cor4.2} are coincide with bounds obtained in \cite[Corollary 1, Corollary 2, Corollary 4 and Corollary 5]{HOG-GMS-JS-Fib-2018}. 
\end{remark}

\end{document}